\documentclass[12pt]{amsart}
\usepackage[colorlinks=true, pdfstartview=FitV, linkcolor=blue, citecolor=green, urlcolor=black,filecolor=magenta]{hyperref}

\usepackage{graphicx, verbatim,amsmath,amssymb}

\newcommand{\be}{\begin{equation}}
\newcommand{\ee}{\end{equation}}

\newcommand{\T}{{\bf T}}
\newcommand{\R}{{\mathbb R}}

\newcommand{\Z}{{\mathbb Z}}

\newtheorem{thm}{Theorem}[section]
\newtheorem*{thm*}{Theorem}

\newtheorem{prop}[thm]{Proposition}

\newtheorem{con}[thm]{Conjecture}

\everymath{\displaystyle}

\theoremstyle{definition}

\begin{document}
\title{Pattern Equivariant Mass Transport in Aperiodic Tilings and Cohomology}
\author{Michael Kelly and Lorenzo Sadun}

\address{Michael Kelly \\ Center for Communications Research \\ Princeton, NJ 08540} 
\email{mskelly@idaccr.org}
\address{Lorenzo Sadun\\Department of Mathematics\\The University of
 Texas at Austin\\ Austin, TX 78712} 
\email{sadun@math.utexas.edu}

\date{\today}

\begin{abstract}
Suppose that we have a repetitive and aperiodic tiling $\T$ of $\R^n$, and two 
mass distributions $f_1$ and $f_2$ 
on $\R^n$, 
each pattern equivariant with respect
to $\T$. Under what circumstances is it possible to do a bounded transport
from $f_1$ to $f_2$? When is it possible to do this transport in a 
strongly or weakly pattern-equivariant way? We reduce these questions to
properties of the \v Cech cohomology of the hull of $\T$, properties that
in most common examples are already well-understood.  
\end{abstract}

\maketitle

\setlength{\baselineskip}{.6cm}


\section{Introduction and Results}
\label{Intro}

A classic problem of transport can be phrased as follows. Given two
countable and uniformly discrete point sets $X_1$ and $X_2$ in $\R^n$,
does there exist a bijection $b: X_1 \to X_2$ such that the distance
from points $x \in X_1$ to corresponding points $b(x) \in X_2$ is
uniformly bounded?  Such a bijection, with $|b(x)-x|$ uniformly
bounded, is call a {\em bounded transport} from $X_1$ to $X_2$, and
$X_2$ is said to be of {\em bounded displacement (BD)} from $X_1$. The
existence of bounded transport is governed by the Hall Marriage
Theorem and the proof of the Schr\"oder-Bernstein theorem (as in
\cite{Halmos}, \cite{Whyte}).

For any compact subset $U \in \R^n$, let $\|U\|_1$ be the number of
points in $U \cap X_1$ and let $\|U\|_2$ be the number of points in $U
\cap X_2$. Let $|U|$ denote the volume of $U$. For each constant
$r>0$, let $U_r = \{x \in \R^n | d(x,U) \le r\}$ be the closed
neighborhood of radius $r$ around $U$, and let $U_{-r} = \{ x \in U |
d(x, U^c) < r\}$ be the complement of the open neighborhood of radius $r$ around
$U^c$.

\begin{thm}[Hall Marriage Theorem]\label{Hall1}
There exists a bounded transport $b: X_1 \to X_2$ 
with $\sup\{|b(x)-x|\} \le r$ if and only if, for every compact set 
$U \in \R^n$, $\|U_r\|_1 \ge \|U\|_2$ and $\|U_r\|_2 \ge \|U\|_1$. 
\end{thm}

An important special case is where $X_1$ has a well-defined density
$\rho$ and where $X_2$ is a lattice of the same density. In that case,
Laczkovich \cite{Laczkovich1,Laczkovich2} (see also \cite{DKLL,SodinTsirelson,Whyte}) showed that

\begin{thm}\label{L1} If $n=2$, then $X_1$ is BD to a lattice if and
  only if there
exist constants $c_1$ and $c_2$ such that, for 
all topological disks 
$U$, $\big | \|U\|_1 - \rho |U| \big | 
\le c_1 + c_2 |\partial U|$, where $|U|$ is the
area of $U$ and $|\partial U|$ is the perimeter of $U$.
\end{thm}
(A similar theorem applies to $n>2$, with two small adjustments. 
$U$ must be a topological ball, and one must either
restrict $U$ to a union of unit cubes with vertices at integer points, 
or replace $|\partial U|$ with $|U_1|-|U_{-1}|$.)

A simple generalization is where the discrete point sets $X_1$ and
$X_2$ are replaced by continuous mass distributions on $\R^n$.  If
$f_1$ and $f_2$ are non-negative functions in $L^1_{loc}(\R^n)$, we
can seek a non-negative function
$b \in L^1_{loc}(\R^n \times \R^n)$ and
a constant $r$ such that
$$ \int_{\R^n} b(x,y) dy = f_1(x), \qquad \int_{\R^n} b(x,y) dx = f_2(y),
\qquad b(x,y)=0 \hbox{ when } |x-y|>r. $$
Theorems similar to (\ref{Hall1}) and (\ref{L1}) are well-known.   

In this paper we impose restrictions on the point patterns $X_1$ and
$X_2$, or the continuous distributions $f_1$ and $f_2$. Given a
repetitive and aperiodic tiling $\T$ of $\R^n$ and two positive
strongly pattern equivariant (PE) mass distributions\footnote{See Section 2
for precise definitions of strong and weak pattern equivariance.}  
$f_1$ and $f_2$,
we ask:
\begin{enumerate}
\item When does there exist bounded transport from $f_1$ to $f_2$?
\item When is it possible to do this transport in a weakly PE way?
\item When is it possible to do this transport in a strongly PE way?
\end{enumerate}
Strongly PE transport is automatically weakly PE, and weakly PE
transport is automatically bounded, but do there exist mass
distributions $f_{1,2}$ that admit bounded transport without admitting
weakly PE transport, or that admit weakly PE transport without
admitting strongly PE transport?

For instance, consider a 1-dimensional Fibonacci tiling, generated by
the substitution $a \to ab$, $b \to a$, with each $a$ tile having length
$\phi = (1+\sqrt{5})/2$ and each $b$ tile having length 1. Let $f_1$ assign
mass 1 to every $a$ tile and mass 0 to every $b$ tile. Let $f_2$ assign mass
0 to every $a$ tile and $\phi$ to every $b$ tile. These distributions have the
same density, but is there a (bounded, weakly PE, or strongly PE) transport 
from one to the other?

\smallskip

\begin{figure}[htp] \label{fig1}
\begin{center}
 \includegraphics[scale=0.5]{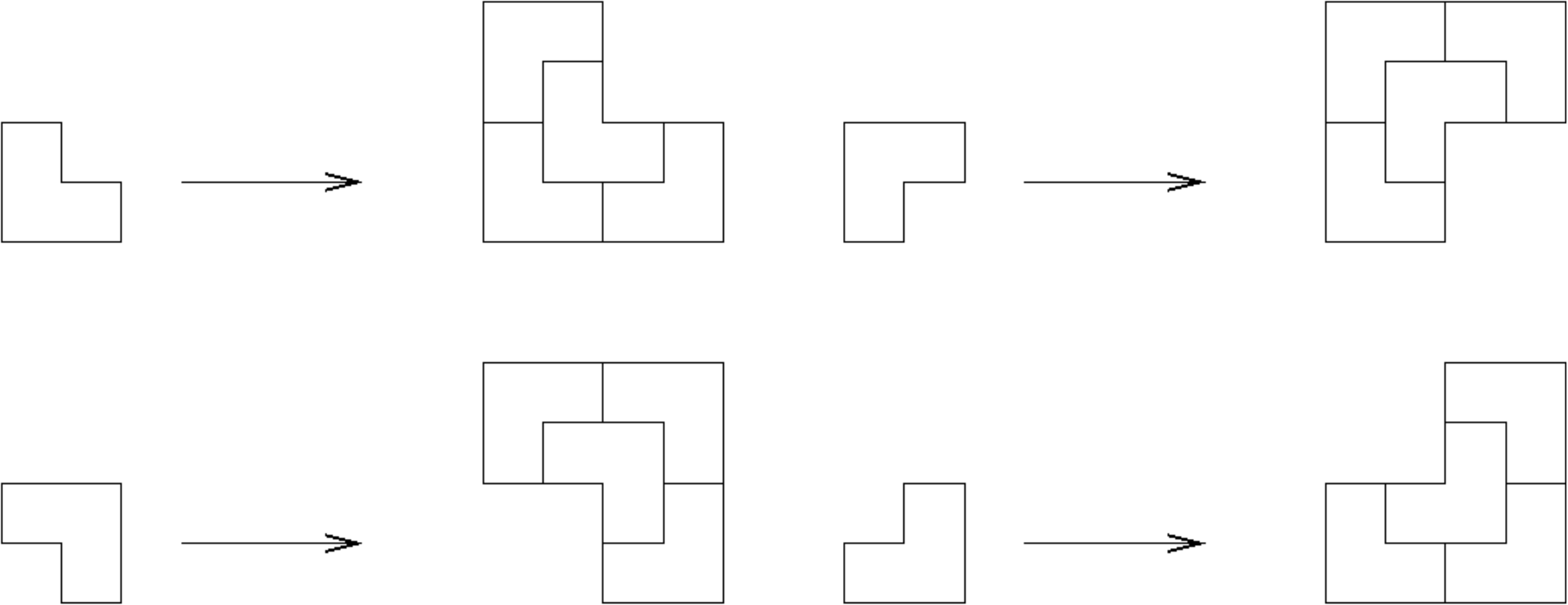}
 \caption{The chair substitution}
\end{center}
\end{figure}

\smallskip

A second example involves the 2-dimensional chair substitution, illustrated
in Figure \ref{fig1}. There are four species of L-shaped tiles, which we 
label by the piece of the $2 \times 2$ square that is missing. That is, the
first tile listed in Figure \ref{fig1} is NE (northeast), the second
is SE, the third is SW, and the last is NW.  
We consider three mass distributions, shown in Figure \ref{fig2}:
\begin{enumerate}
\item $f_1$ assigns mass 2 to NE tiles, and mass 0 to NW, SW, and SE tiles.
\item $f_2$ assigns mass 1 to NE and SW tiles, and mass 0 to NW and SE tiles.
\item $f_3$ assigns mass 0 to NE and SW tiles, and mass 1 to NW and SE tiles.
\end{enumerate}
As before, we ask which transports between $f_1$, $f_2$ and $f_3$ can be done
in a bounded, weakly PE, or strongly PE manner. 
  
\smallskip

\begin{figure}[htp] 
\begin{center}
 \includegraphics[scale=0.25]{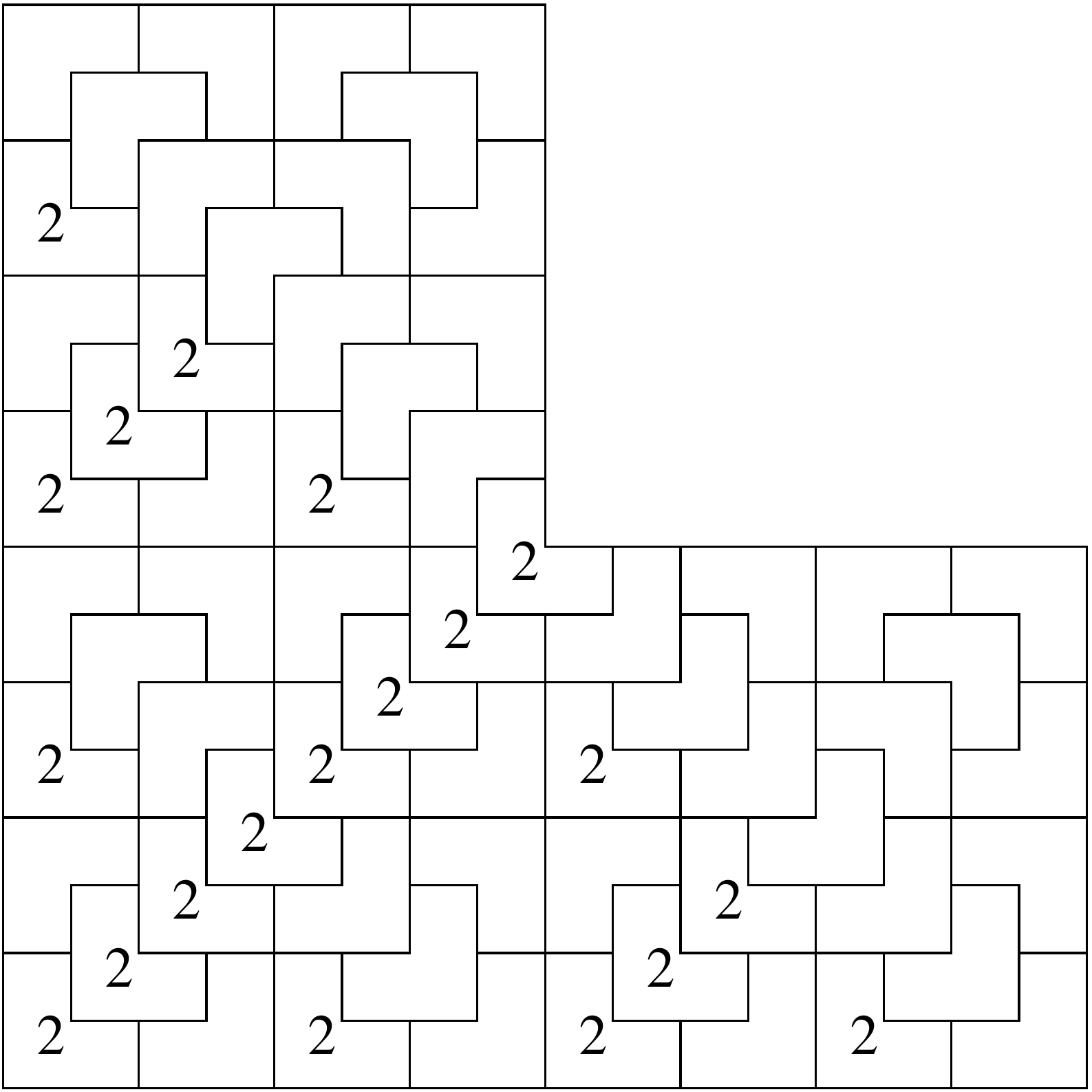} \qquad
\includegraphics[scale=0.25]{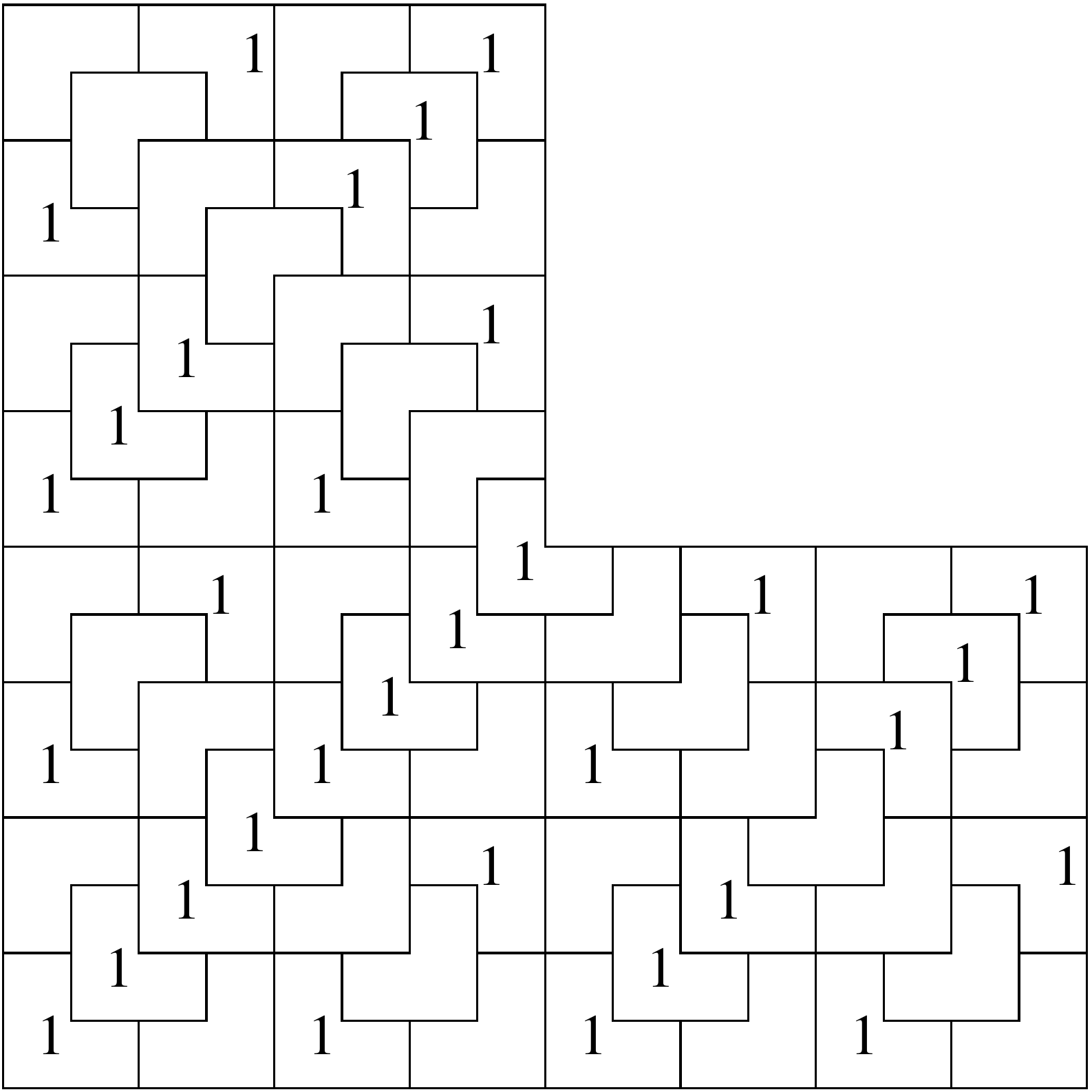} \qquad
\includegraphics[scale=0.25]{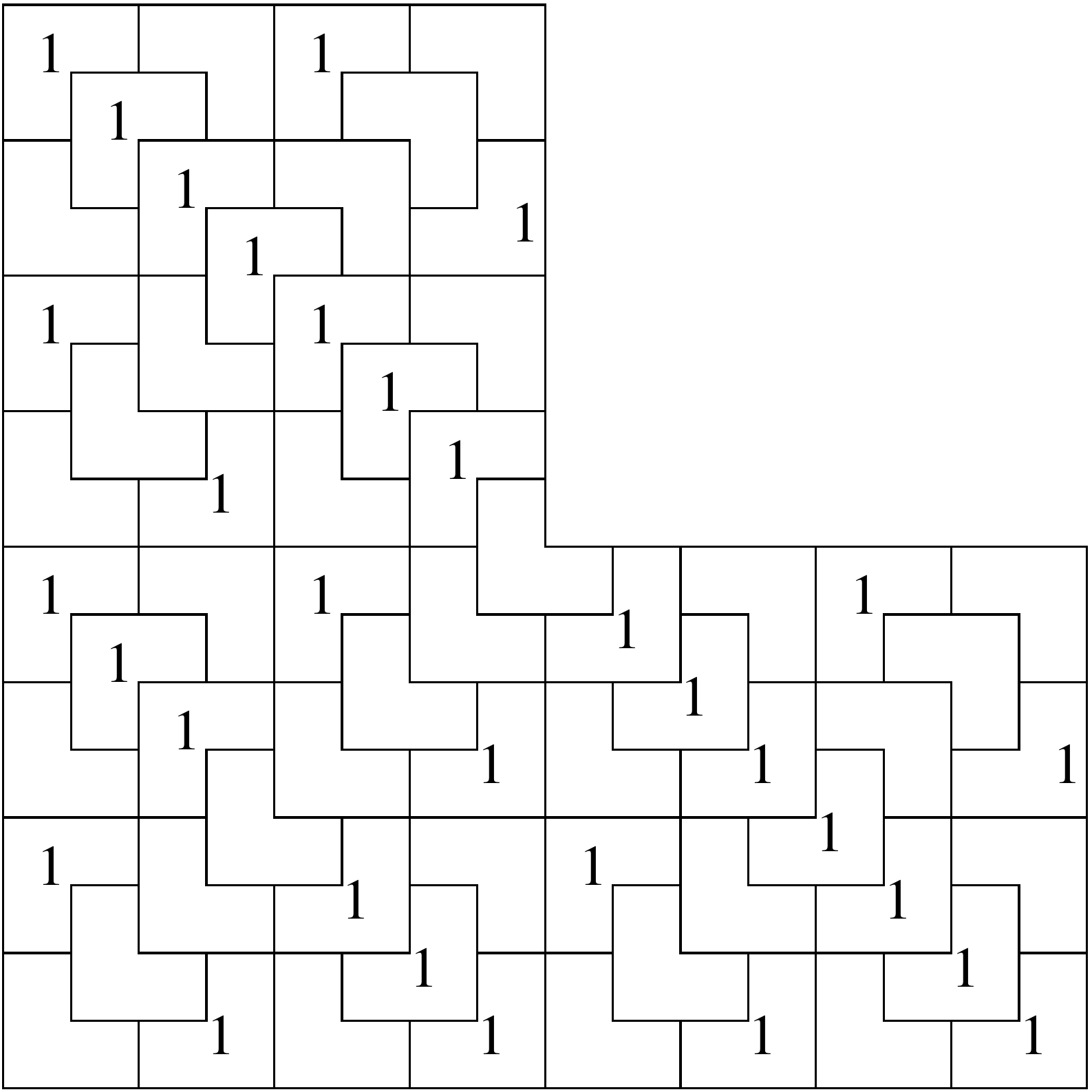}
\caption{Three different mass distributions on a patch of the chair tiling}
\label{fig2}
\end{center}
\end{figure}

\smallskip

We convert these questions to questions of cohomology. Since much is
already known about the cohomology of tiling spaces
(\cite{AP,FHK,KP,SadunInteger,SadunBook,SadunNewChapter,SadunWilliams}), this reduces many
problems of transport either to a simple look-up or to calculations
using well-established techniques \cite{CS,CS2,KS,KS2,KellySadun,SadunHierTilings,Sadun1D}.

We associate a class $[f_i] \in \check H^n(\Omega_\T, \R)$, the top
real-valued \v Cech cohomology of the continuous hull $\Omega_\T$ of
$\T$, to each mass distribution $f_i$. We will define subspaces of
$\check H^n(\Omega_\T, \R)$, called ``asymptotically negligible'' and
``well-balanced'' classes, and show that:

\begin{thm}\label{Thm1} 
  Let $[f_1]$ and $[f_2]$ be the classes in $\check H^n(\Omega,\R)$
  associated to the strongly PE mass distributions $f_1$ and
  $f_2$. Then
\begin{enumerate}
\item There exists a bounded transport from $f_1$ to $f_2$ if and only if 
$[f_1]-[f_2]$ is well-balanced. 
\item There exists a weakly PE transport from $f_1$ to $f_2$ if and only if 
$[f_1]-[f_2]$ is asymptotically negligible.
\item There exists a strongly PE transport from $f_1$ to $f_2$ if and only
if $[f_1]=[f_2]$. 
\end{enumerate}
\end{thm}

There are many examples of non-zero classes that are asymptotically
negligible, leading to mass distributions that admit weakly PE
transport but not strongly PE transport. In particular, we will see
that the Fibonacci example admits weakly PE but not strongly PE
transport from $f_1$ to $f_2$.

\begin{con} 
Let $\T$ be an arbitrary tiling that is repetitive and has finite
local complexity. Then every well-balanced class in $H^n(\Omega_\T,
\R)$ is asymptotically negligible.
\end{con}

In a subsequent paper we will address this conjecture, and its
generalization to classes in $H^k(\Omega_\T, \R)$ with $k<n$, in the
context of substitution tilings \cite{vaporware}.

In this paper we prove
\begin{thm}\label{Thm2} 
Let $\T$ be a repetitive and aperiodic tiling of $\R^n$. If
\begin{itemize}
\item $n=1$, or 
\item $\T$ is a codimension-1 cut-and-project tiling with canonical window, 
\end{itemize}
then every well-balanced class in $\check H^n(\Omega_\T, \R)$ is 
asymptotically negligible. In particular, given strongly PE mass distributions
$f_1$ and $f_2$, there exists 
a bounded transport from $f_1$ to $f_2$ if and only
if there exists a weakly PE transport from $f_1$ to $f_2$. 
\end{thm}

\medskip

In Section \ref{Background} we review the formalism of aperiodic
tilings, continuous hulls, \v Cech cohomology and PE cohomology, and
we precisely define what it means for a class in $\check
H^n(\Omega_\T, \R)$ to be well-balanced or asymptotically
negligible. In Section \ref{Pictures} we give a cohomological
interpretation of the problem of PE transport, and prove Theorem
\ref{Thm1}. We then use Theorem \ref{Thm1} to solve the Fibonacci and
chair examples. In Section \ref{AN=WB} we prove a slightly more
general version of Theorem \ref{Thm2}.

\section*{Acknowledgments}

We thank Jean-Marc Gambaudo and Francois Gautero for introducing us to
this problem several years ago, and for helpful discussions.  We also
congratulate J-M.G.~on the occasion of his 60th birthday. The work of
the second author is partially supported by NSF
  grant  DMS-1101326

\section{Background} \label{Background}

In this section we go over the general formalism of tilings and tiling
cohomology. For details, see (\cite{SadunBook}, Chapters 1 and 3). 

A {\em tile} is a topological ball in $\R^n$, equal to the closure of 
its interior, and equipped with a label. A {\em tiling} is a 
set of tiles whose union is all of $\R^n$, such that tiles only intersect on
their boundaries. If $t$ is a tile and $x \in \R^n$, then $t-x$ is a tile,
with the same label as $t$, obtained by translating all the points of $t$ by
$-x$. Two tiles $t_1$ and $t_2$ are {\em translationally equivalent} if
$t_2= t_1 - x$ for some $x \in \R^n$. The equivalence classes of this
relation are called {\em prototiles}. A {\em patch} is a finite set of 
tiles in a tiling, and $\R^n$ acts on patches by moving each tile separately.

If $K$ is a compact set
and $\T$ is a tiling, let $[K]_\T$ denote the patch of tiles in $\T$ that
intersects $K$. 
We assume that our tilings have {\em finite local complexity}, or FLC.
This means that for any $K$, the set $\{ [K-x]_\T | x \in \R^n\}$
is finite up to translation. This is equivalent to 
there are being finitely many prototiles, and finitely many
ways that two tiles can touch, up to translation. 

Let $B_R(x)$ denote the closed ball of radius $R$ around a point $x \in \R^n$.
If $\T_1$ and $\T_2$ are two FLC tilings, we say that $\T_2$ is {\em locally
derivable} from $\T_1$ if there exists a radius $R$ such that 
$[B_R(0)]_{\T_1-x} = [B_R(0)]_{\T_1-y}$, then $[B_1(0)]_{\T_2-x} = [B_1(0)]_{\T_2-y}$.
That is, the pattern of $\T_2$ in a ball of radius 1 
around $x$ (i.e. the pattern of $\T_2-x$ in a ball of radius 1 around the origin) is determined exactly
by the pattern of $\T_1$ is a ball of radius $R$ around $x$. If $\T_2$ is 
locally derivable from $\T_1$ and $\T_1$ is locally derivable from $\T_2$,
we say that $\T_1$ and $\T_2$ are {\em mutually locally derivable}, or MLD.

We can also speak of labeled point patterns being 
locally derivable from tilings,
or vice-versa. A discrete point pattern $X$ (with each point being 
assigned one of a finite set of labels) is locally derived from $\T$
if there exists an $R>0$ such that  
$[B_R(0)]_{\T-x} = [B_R(0)]_{\T-y}$, then $B_1(0) \cap {X-x} = B_1(0) \cap {X-y}$.
Conversely, $\T$ is locally derived from $X$ if there exists an $R>0$ such
that, whenever
$B_R(0) \cap {X-x} = B_R(0) \cap {X-y}$,
$[B_1(0)]_{\T-x} = [B_R(1)]_{\T-y}$. A point pattern $X_2$ is locally derived
from $X_1$  if there exists an $R>0$ such
that, whenever
$B_R(0) \cap {X_1-x} = B_R(0) \cap {X_1-y}$,
$B_1(0) \cap {X_2-x} = B_1(0) \cap {X_2-y}$. 

Given a tiling, the set of vertices of that tiling, with appropriate labels
for the vertices, is MLD to the original tiling. Given a point pattern, 
the set of Voronoi cells of that point pattern, with appropriate vertices,
is MLD to the point pattern. By combining these two operations, we see that
every FLC tiling is MLD to a tiling
by convex polytopes that meet full-face to full-face. 

Given a set of prototiles, there is a metric on the space of tilings
by those prototiles. Two FLC tilings $\T_1$, $\T_2$ are $\epsilon$-close
if there exist $x_1, x_2 \in \R^n$ with $|x_i| < \epsilon$, such that
$[B_{1/\epsilon}(0)]_{\T_1-x_1} = [B_{1/\epsilon}(0)]_{\T_2-x_2}$. That is, two tilings 
are close if they agree on a large ball up to a small translation. The 
{\em continuous hull} $\Omega_\T$ of a tiling $\T$ is the completion of the
translational orbit $\{ \T-x\}$. A tiling $\T'$ is in $\Omega_\T$ if and
only if every patch of $\T'$ is a translate of a patch of $\T$. 

An FLC tiling $\T$ is {\em repetitive} if for every patch $P$ of $\T$ there
exists a radius $R$ such that every ball of radius $R$ in $\T$ contains at
least one occurrence of $P$. An FLC tiling $\T$ is repetitive if and only if
the hull $\Omega_\T$ is a minimal dynamical system, i.e. if all tilings
in $\Omega_\T$ exhibit the same set of patches (up to translation). 

A tiling $\T$ is {\em aperiodic} if $\T-x = \T$ implies $x=0$. If $\T$
is aperiodic and repetitive, then a neighborhood of $\T$ in $\Omega_\T$ 
is homeomorphic to the product of a Cantor set and an open subset of $\R^n$.

{\bf We henceforth assume that all tilings are FLC, aperiodic and
repetitive, with tiles that are convex polytopes that meet full-face to 
full-face.}  

Given such a tiling $\T$, a continuous function $f: \R^n \to \R$ is said
to be {\em strongly pattern equivariant} (PE) with respect to $\T$ if 
there exists a radius $R$ such that, whenever $[B_R(0)]_{\T-x}=
[B_R(0)]_{\T-y}$, $f(x)=f(y)$. That is, $f$ is PE with radius $R$ if the
value of $f(x)$ is determined exactly 
by the pattern of $\T$ in a ball of radius 
$R$ around $x$. A {\em weakly PE} function is the uniform limit of 
strongly PE functions of arbitrary radius. A function $f$ is weakly PE
if for every $\epsilon > 0 $ there exists a radius $R_\epsilon$ such that 
the value of $f(x)$ is determined to within $\epsilon$ by the pattern of 
$\T$ in a ball of radius $R_\epsilon$ around $x$. 

The tiling $\T$ gives a decomposition of $\R^n$ into 0-cells (vertices),
1-cells (edges), 2-cells (faces), etc. A $k$-cochain is an assignment of 
a real number to each $k$-cell. As with functions, $k$-cochains can be 
weakly or strongly PE. Let $\Omega^k(\T)$ be the set of strongly PE 
$k$-cochains. The coboundary of a strongly PE cochain is strongly
PE (albeit possibly with a slightly larger radius), yielding a complex  
\be 0 \rightarrow \Omega^0(\T) \xrightarrow{\delta_0} \Omega^1(\T) 
\xrightarrow{\delta_1} \cdots \xrightarrow{\delta_{n-1}} \Omega^n(\T) 
\rightarrow 0.
\ee

Another version of pattern equivariant cohomology (in fact, the
original version proposed in \cite{KP}) uses differential forms.  We
say that a $k$-form $\alpha$ on $\R^n$ is strongly PE if it can be
written as a sum $\sum_I \alpha_I(x) dx^I$ and each coefficient
function $\alpha_I$ is strongly PE.  We say that $\alpha$ is weakly PE
if each $\alpha_I$, and all derivatives of $\alpha_I$ of all orders,
are weakly PE functions. The exterior derivative of a strongly PE form
is strongly PE, so we can consider the complex of strongly PE forms.
The key fact, due to \cite{KP} for forms and to \cite{SadunInteger} for
cochains, is:

\begin{thm}\label{Cech} The cohomology of the de-Rham-like complex of  
  strongly PE forms on $\T$, and the cohomology of the complex of
  real-valued PE cochains, are both canonically isomorphic to the
  real-valued \v Cech cohomology of the continuous hull $\Omega_\T$.
\end{thm}

Note that both versions of PE cohomology depend only on the tiling
space $\Omega_\T$ and not on the particular tiling $\T$ that is used
for the calculations, and both are invariant under homeomorphisms of
$\Omega_\T$.  Thanks to this fact, we can use the symbol $\check
H^k(\Omega_\T, \R)$ to refer the real-valued PE cohomology of $\T$
(using either forms or cochains), as well as to the \v Cech cohomology
of $\Omega_\T$.  Moreover, the isomorphism of form-based and
cochain-based cohomology is easy to construct. Every class in the
form-based cohomology is represented by a closed PE $k$-form, which
can be integrated over $k$-cells to give a closed PE cochain, which
represents a class in the cochain-based cohomology.  In particular, if
we have a smooth strongly PE mass density $f$, then we have a strongly
PE $n$-form $f d^nx$.  Integrating this form over tiles gives a
strongly PE $n$-cochain.

Suppose that $\alpha$ is a strongly PE $n$-cochain. For dimensional
reasons, $\delta \alpha =0$, so $\alpha$ represents a cohomology class
in $\check H^n(\Omega_\T, \R)$. We call such a cochain {\em well
  balanced} if there exists constants $c_1$ and $c_2$ such that, for
every patch $P$,
\begin{equation}
\left | \int_P \alpha \right | \le c_1 + c_2 |\partial P|,
\end{equation}
where $|\partial P|$ is the $(n-1)$-dimensional Lebesgue measure of
the boundary of $P$, and where $\int_P \alpha$ denotes the value of
$\alpha$ applied to the chain $P$. If $\alpha = \delta \beta$ for some
strongly PE (and therefore bounded) $(n-1)$-cochain $\beta$, then this
condition is always met, since
\begin{equation} \int_P \alpha = \int_{\partial P}  \beta. \end{equation}
The well-balanced condition therefore only
depends on the cohomology class of $\alpha$, and we define $\check
H^n_{WB}(\Omega_\T, \R)$ to be the classes in $\check H^n(\Omega_\T,
\R)$ represented by well-balanced cochains.

If $\alpha$ is a strongly PE $n$-cochain and there exists a {\em
  weakly} PE $(n-1)$-cochain $\beta$ such that $\alpha = \delta
\beta$, then we say that $\alpha$ is {\em weakly exact}.  As before,
this only depends on the cohomology class of $\alpha$, since if
$\alpha' = \alpha + \delta \gamma$, where $\gamma$ is a strongly PE
cochain, and if $\alpha = \delta \beta$, where $\beta$ is weakly PE,
then $\alpha' = \delta (\gamma + \beta)$, where $\gamma + \beta$ is
weakly PE. The cohomology class of a weakly exact cochain is said to
be {\em asymptotically negligible}.  We denote the asymptotically
negligible classes in $H^n(\Omega_T, \R)$ by $\check
H^n_{AN}(\Omega_\T, \R)$.

\begin{prop}$\check H^n_{AN}(\Omega_\T, \R) \subset \check 
H^n_{WB}(\Omega_\T, \R)$.
\end{prop}

\begin{proof} 
  Suppose that $\alpha=
  \delta \beta$ with $\beta$ weakly PE. Since $\beta$ is the uniform
  limit of strongly PE cochains, and since strongly PE cochains are
  bounded, $\beta$ is bounded. That is, there exists a constant $C$
  such that, for every $(n-1)$-cell $z$ in $\T$, $|\beta(z)| \le C
  |z|$. But then, for every patch $P$,
$$ \left | \int_P \alpha \right | = \left | \int_{\partial P} \beta \right | \le C | \partial P|,$$
so $\alpha$ is well-balanced.
\end{proof}

\medskip

We complete this section by considering what it means for bounded transport 
to be strongly or weakly PE. The situation is slightly different, depending
on whether we consider points, masses concentrated at points,  
or countinuous distributions of mass. 

There is no such thing as weakly PE transport of points.  If $X_1$ and
$X_2$ are discrete point patterns, each locally derived from $\T$,
then a bounded transport $b: X_1 \to X_2$ is {\em strongly} PE if
there exists an $R>0$ such that, whenever $x,y \in X_1$ and $B_R(0)
\cap (X_1-x) = B_R(0) \cap (X_1-y)$, then $b(x)-x = b(y)-y$. That is,
the displacement $b(x)-x$ is determined exactly by the pattern of
$X_1$ in a ball around $x$, which in turn is determined exactly by the
pattern of $\T$ in a ball (of possibly bigger radius $R'$) around
$x$. However, since $X_2$ has FLC (being locally derived from the FLC
tiling $\T$), it is impossible to approximate one bounded transport by
another, so it does not make sense to speak of the bounded transport
$b$ being weakly PE.

However, there is a more general notion of mass transfer among point
masses where weak pattern equivariance does make sense.  Suppose that
$b: X_1 \times X_2 \to \R$, and that for $x_1 \in X_1$ and $x_2 \in
X_2$, $b(x_1, x_2)$ represents the amount of mass transported from
$x_1$ to $x_2$. This needn't be an integer. All that is required is
that $0 \le b(x_1, x_2) \le 1$, that for fixed $x_1$ we have
$\sum_{x_2 \in X_2} b(x_1, x_2) = 1$, and that for fixed $x_2 \in X_2$
we have $\sum_{x_1 \in X_1} b(x_1, x_2) = 1$.  The transport $b$ is
{\em bounded} if there is a constant $R$ such that $b(x_1, x_2) = 0$
whenever $|x_1-x_2|>R$. The bounded transport $b$ is strongly PE if
there is a constant $R'$ such that $b(x_1, x_2)$ is determined exactly
by $B_{R'}(0) \cap (X_1-x_1)$, and $b$ if weakly PE is $b$ is the
uniform limit of strongly PE transport functions.

We also consider mass transport from one density $f_1$ to another
density $f_2$. A transport function is a function $b: \R^n \times \R^n
\to \R$, where $b(x,y)$ is a {\em density} of mass transported from
$x$ to $y$.  More precisely, if $A, B \subset \R^n$ are compact sets,
then $\iint_{A \times B} b(x,y) d^nx \, d^n y$ is the amount of mass
transported from $A$ to $B$. As with mass distributions localized at
points, $b$ is bounded if there is a constant $R$ such that $b(x,y)=0$
whenever $|x-y| > R$. The transport $b$ is strongly PE if $b(x,y)$ is
determined exactly by the pattern of $\T$ is a ball of radius $R'$
around $x$ (equivalently, by $[B_{R'}(0)]_{\T-x}$). This can be
expressed in an integral form.  Let $b_{A,B}(x,y) = \iint_{(x+A)\times
  (y+B)} b(x', y') d^n x' \, d^n y'$. This is the amount of mass
transported from $x+A$ to $y+B$. $b$ is a strongly PE transport
density if and only if, for arbitrary $A$ and $B$, $b_{A,B}$ is a
strongly PE function.  We say that $b$ is a weakly PE transport
density if and only if, for arbitrary $A$ and $B$, $b_{A,B}$ is a
weakly PE function.

\section{The Cohomological Picture}\label{Pictures}

Suppose that $X_1$ and $X_2$ are discrete point patterns, each locally
derived from a non-periodic, repetitive FLC tiling $\T$.  Note that
$X_1$ and $X_2$ are Delone sets.  Let $R$ be a distance such that
every ball of radius $R$ contains at least one point from each of the
two collections. Let $\rho$ be a bump function whose support is a ball
of radius $R$. Let $f_i(x) = \sum_{x' \in X_i} \rho(x-x')$.  This can
be viewed as the convolution of $\rho$ with a Dirac comb supported on
$X_i$.

\begin{thm} \label{Thm-point} If there is a bounded transport from
  $X_1$ to $X_2$ in the sense of point {\em masses}, then there is a
  bounded transport from $f_1$ to $f_2$ in the sense of density
  functions. Furthermore, if there exists a transport from $X_1$ to
  $X_2$ that is (strongly or weakly) PE, then then there exists a
  similarly PE transport from $f_1$ to $f_2$.
\end{thm}
\begin{proof} The strategy for getting a transport from $f_1$ to $f_2$
  is to first move mass by up to $R$ to convert $f_1$ to a collection
  of unit point masses at $X_1$, then to do the bounded transport from
  $X_1$ to $X_2$, and then to move mass by up to $R$ to get $f_2$.
 
  Specifically, if $b_0: X_1 \times X_2 \to \R$ describes the
  transport from $X_1$ to $X_2$, then for $x,y \in \R^n$ we define
\begin{equation}
b(x,y) = \sum_{x' \in X_1} \sum_{y' \in X_2} \rho(x-x') \rho(y-y') b_0(x', y').
\end{equation}
This sum is finite, since it only involves points $x'$ within a distance $R$ of $x$ and points $y'$ within a distance $r$ of $y$. 
If $b_0(x',y')=0$ whenever $|x'-y'| > R'$, then $b(x,y)=0$ whenever $|x-y| > 2R+R'$. Note that 
\begin{eqnarray} 
\int_{\R^n} b(x,y) d^n y &=& \sum_{x' \in X_1} \sum_{y' \in X_2} \rho(x-x') b_0(x',y') \cr 
& = &  \sum_{x' \in X_1} \rho(x-x') \cr 
&=& f_1(x),
\end{eqnarray}
and similarly $\int b(x,y) d^nx = f_2(y)$. If $b_0$ is strongly PE with radius $R''$, then $b$ is strongly PE with 
radius $2R + R''$. If $b_0$ is weakly PE, then $b$ is weakly PE.  
\end{proof}

The reverse implication is less immediate, since there isn't a canonical 
way to convert a continuous distribution into a point pattern. We'll return
to this question later. 

We now relate transport of continuous distributions to cohomology. If
$f_1$ and $f_2$ are strongly PE mass densities, let $\alpha_i$ be a
(strongly PE) cochain whose value on a tile is the integral of $f_i
d^nx$ over that tile. Let $[f_1]$ and $[f_2]$ be the cohomology
classes of $\alpha_1$ and $\alpha_2$, respectively.

\begin{thm}\label{two-three}
 There is a bounded transport from $f_1$ to $f_2$ 
if and only if there exists a bounded $(n-1)$-cochain $\beta$ such that
$\delta \beta = \alpha_1-\alpha_2$. Furthermore, the transport 
from $f_1$ to $f_2$ can be chosen to be (strongly or weakly)
PE if and only if $\beta$ can be chosen to be (strongly or weakly) PE. 
\end{thm}
 
\begin{proof}
If there is a transport from $f_1$ to $f_2$, let $\beta$ on
a tile face be the minus the 
net mass transported across that face (say along a 
straight line).  $\delta \beta$ applied to a tile is minus 
the sum of the transfers
on all of its faces, i.e. minus the change in mass, so $\delta \beta
= \alpha_1-\alpha_2$. 
If the transport is bounded by a distance $R$, then the
mass transported across a face is bounded by the total existing mass within
$R$ of that face, so $\beta$ is bounded. Likewise, if the transport is
pattern-equivariant, then so is $\beta$. 

For the converse, we assign a point $p(t)$ to each tile $t$ (in a
strongly PE way) and do a bounded transport to put all the mass
$\alpha_1(t)$ of the tile $t$ at that point. Given a bounded cochain
$\beta$ with $\delta \beta = \alpha_1 - \alpha_2$, we will construct a
bounded transport that converts this collection of point masses into a
collection of point masses of size $\alpha_2(t)$ at each $p(t)$. By
Theorem \ref{Thm-point} (or more precisely, by the proof of the
theorem adapted to the case where the points do not have unit mass),
this implies the existence of a bounded transport from $f_1$ to $f_2$.

Since each $f_i$ is positive, there is a lower bound $\epsilon>0$ to
the value of $\alpha_i$ on any tile. There is also an upper bound
$N_1$ to the number of faces a tile can have.  Let $N_2$ be an upper
bound to the values of $|\beta(c)|$, and let $N_3 = \lceil MN/\epsilon
\rceil$, so that $|\beta(c)|/N_3 < \epsilon/N_1$.  We do our mass
transport in $N_3$ steps, in each step transferring mass
$-\beta(c)/N_3$ across the face $c$.  At every step, the mass in each
tile $t$ is a weighted average of $\alpha_1(t)$ and $\alpha_2(t)$, and
so is at least $\epsilon$, so there is enough mass in each tile to do
the transfer.  At each step, each piece of mass is moved a distance at
most twice the diameter of the largest tile, so the $N_3$-step
transfer process is bounded.

If $\beta$ is (weakly or strongly) PE, then the transfer at each step
from point masses of size $\alpha_1(t)$ to point masses of size
$\alpha_2(t)$ is (weakly or strongly) PE, making the entire process
(weakly of strongly) PE.
\end{proof}

\bigskip

\begin{proof}[Proof of Theorem \ref{Thm1}]

  Statement 1: If there exists a bounded $\beta$, then $[f_1]-[f_2]$
  is well-balanced, since for any patch $P$, $\left | \int_P (f_1(x) -
    f_2(x)) d^n x \right |= \left | \int_P \alpha_1 - \alpha_2 \right
  | = \left | \int_{\partial P} \beta \right | \le \hbox{const. }
  | \partial P|.$

  Conversely, if $[f_1]-[f_2]$ is well-balanced, then $\left | \int_P
    (f_1-f_2) d^nx \right | = \left | \int_P \alpha_1 - \alpha_2
  \right |$ is bounded by a fixed multiple of $|\partial P|$. Since
  there are lower bounds to the densities $f_1$ and $f_2$, there is a
  constant $r_0$ (independent of $P$) such that the integrals of $f_1$
  and $f_2$ over every region of area at least $r_0 |\partial P|$ are
  both greater than $\left | \int_P (f_1 - f_2) d^nx \right |$. This
  implies that there is an $r$ (of the same order as $r_0$, but whose
  precise derivation requires some geometrical arguments) such that
  the integral of $f_1$ over every $r$-neighborhood of $P$ is greater
  than the integral of $f_2$ over $P$, and such that the integral of
  $f_2$ over the $r$-neighborhood is greater than the integral of
  $f_1$ over $P$. By the Hall Marriage Theorem (see also \cite{Laczkovich1})
  this implies the existence of a bounded transport from $f_1$ to
  $f_2$.

\smallskip

Statement 2: By Theorem \ref{two-three}, there exists a weakly PE
transport from $f_1$ to $f_2$ if and only if there exists a weakly PE
cochain $\beta$ such that $\delta \beta = \alpha_1 - \alpha_2$. But by
definition, that is the same as $[f_1]-[f_2]$ being asymptotically
negligible.

 \smallskip
 
 Statement 3: By Theorem \ref{two-three}, there exists a strongly PE
 transport from $f_1$ to $f_2$ if and only if there exists a strongly
 PE cochain $\beta$ such that $\delta \beta = \alpha_1 -
 \alpha_2$. But by definition, that is the same as $\alpha_1$ being
 cohomologous to $\alpha_2$, in other words of $[f_1]-[f_2]$ being
 zero.
\end{proof}

\medskip

We now return to our example problems. In the Fibonacci tiling, 
$H^1(\Omega_T, \R)=\R^2$. $H^1_{AN}$ is the contractive subspace of 
$H^1$ under substitution \cite{CS}, which is 1-dimensional, 
so $H^1(\Omega_T,\R) = \R \oplus H^1_{AN}(\Omega_T,\R)$. In particular, any
two classes with the same overall density differ by an asymptotically 
negligible class, so $[f_1-f_2]$ is asymptotically negligible, so 
it is possible to do a weakly PE transport from $f_1$ to $f_2$. 

However, it is not possible to do a strongly PE transport from $f_1$ to 
$f_2$. To see this, suppose that $f_1-f_2 = \delta \beta$, where $\beta$
is a strongly PE function on vertices, say with radius $R$. 
By repetitivity, there are two vertices $v_1$, $v_2$ whose patterns agree
to radius $R$. But then 
\begin{equation}
\int_{v_1}^{v_2} f_1 - \int_{v_1}^{v_2} f_2 = \int_{v_1}^{v_2} \delta \beta
= \beta(v_2)-\beta(v_1)=0.
\end{equation} 
However, $\int_{v_1}^{v_2} f_1$ is an integer, while $\int_{v_1}^{v_2} f_1$ is 
a multiple of $\phi$, so these integrals cannot be equal. Contradiction. 

\smallskip

Next consider the chair tiling. $H^2(\Omega_T,\R)$ is known to equal
$\R^3$, and the three generators are described as follows. For $t\in \{
NE, NW, SE, SW\}$, let $i_t$ be the indicator function of that tile, i.e.
a cochain that evaluates to 1 on every $t$ tile and 0 on the other three
kinds. 

\begin{enumerate}
\item One generator is the constant cochain $i_{NE}+i_{NW}+i_{SE}+i_{SW}$
that simply counts tiles. This generator quadruples under substitution.
\item One generator is $i_{NE}-i_{SW}$, which is the same as $f_1-f_2$. 
It doubles under substitution, and we will soon see that it is neither
WE nor WB. 
\item The third generator is a rotated version of the second, namely
$i_{NW}-i_{SE}$. It, too, is neither WE nor WB. 
\item The cochain $i_{NE}+i_{SW}-i_{NW}-i_{SE}$ evaluates to 0 on every 
substituted tile, and is cohomologically trivial.
\end{enumerate}

Since $f_2-f_3 =i_{NE}+i_{SW}-i_{NW}-i_{SE}$ is exact, there exists strongly PE
transport from $f_2$ to $f_3$, namely rearranging the mass within each 
once-substituted tile. Since $f_1-f_2$ (and hence $f_1-f_3$) is not WB, there
is no bounded transport from $f_1$ to $f_2$ (or $f_3$). 

\smallskip

\begin{figure}[htp] 
\begin{center}
 \includegraphics[scale=0.5]{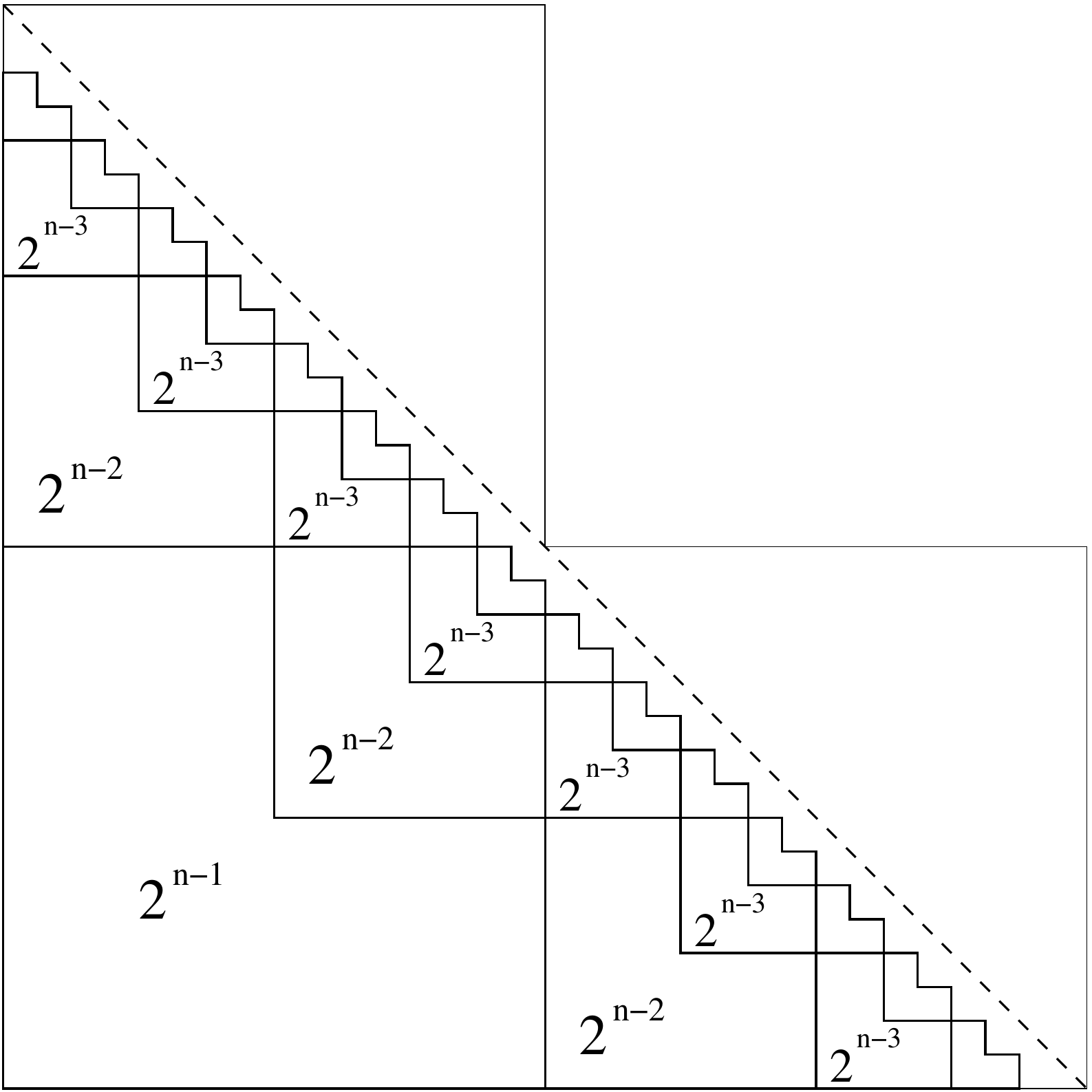}
 \caption{The mass of a partial supertile grows faster than perimeter}
\end{center}\label{partial-supertile}
\end{figure}

\smallskip

To see that $\alpha=i_{NE}-i_{SW}$ is not WB, note that $\alpha$
doubles under substitution, and therefore evaluates to $2^m$ on an
$m$-substituted $NE$ tile, also known as an {\em $m$-supertile}.  That
is, the total mass grows like the perimeter for complete
$m$-supertiles.  Now let $R_n$ be the portion of an $n$-supertile of
type NE obtained by cutting along a diagonal line as in Figure
\ref{partial-supertile} and discarding the tiles that straddle the
dividing line. $R_n$ can be subdivided into one $(n-1)$-supertile,
three $(n-2)$-supertiles, seven $(n-3)$-supertiles, $\ldots$, and
$2^n-1$ ordinary tiles, all of type $NE$. We then have
\begin{equation} \int_{R_n} \alpha = 
2^{n-1} + 3\times 2^{n-2} + \cdots + (2^n-1) \times 2^0 = (n-1)\times 2^n + 1.
\end{equation}
Since $n$ is arbitrary and $|\partial R_n| = 4 \times 2^n$, $\int_{R_n}\alpha$
cannot be bounded by a uniform constant times $|\partial R_n|$. 

Note that we have done more than simply solve our example puzzle. We
have shown that for the chair tiling, $H^2_{WB}(\Omega_\T, \R)$ is
trivial. Given {\em any} two strongly PE mass distributions $f$ and
$g$ on a chair tiling, there exists a bounded transport from $f$ to $g$
if and only if $[f-g]=0$, in which case there exists a strongly PE
transport from $f$ to $g$.

\section{Theorem \ref{Thm2}} 
\label{AN=WB}

In the last section we showed that the question: ``If a bounded
transport exists between two strongly PE mass distributions $f_1$ and
$f_2$, does there necessarily exist a weakly PR transport?'', is
equivalent to ``Is every well-balanced class is $\check H^n(\Omega_\T,
\R)$ asymptotically negligible''? In this section we generalize the
cohomological question and then give positive answers in two 
settings. Theorem \ref{Thm2} then follows as a corollary.

The tiling $\T$ defines a decomposition of $\R^n$ into 0-cells
(vertices), 1-cells (edges), 2-cells (faces), etc. Let $\{c^{(k)}_i\}$
denote the set of $k$-cells of the tiling, and pick an orientation for
each cell. (For vertices and $n$-cells there is a canonical choice of
orientation, but in the intermediate dimensions we must make some
arbitrary choices.) A $k$-chain $A_k$ is a finite linear combination
$\sum_i a_i c^{(k)}_i$, and we define $|A_k| = \sum_i |a_i|
|c^{(k)}_i|$, where $|c^{(k)}_i|$ is the $k$-dimensional Euclidean
measure of $c^{(k)}_i$.  If $\alpha$ is a $k$-cochain, then we write
$\int_{c^{(k)}_i} \alpha$ to denote the value of $\alpha$ on the cell
$c^{(k)}_i$, and $\int_{A_k} \alpha := \sum_i a_i \int_{c^{(k)}_i}
\alpha$.

We say that a strongly PE $k$-cochain $\alpha$ is {\em well-balanced}
(WB) if there exists a constant $K$ such that, for any $k$-chain
$A_k$, $\left | \int_{A_k} \alpha \right | \le K |A_k|$. We say that
$\alpha$ is {\em weakly exact} (WE), and that the class of $\alpha$ is
{\em asymptotically negligible} (AN) if there exists a
weakly PE $(k-1)$-cochain $\beta$ such that $\alpha = \delta
\beta$. These properties were previously defined for $n$-cochains, but
in fact the definitions make sense for any $k$. Stokes Theorem says
that $WE \implies WB$. The question is whether $WB \implies WE$.

\begin{prop} Every WE or WB cochain is closed.  \end{prop}
 
 \begin{proof} If $\alpha$ is WE, then $\alpha = \delta \beta$, 
so $\delta \alpha = \delta^2 \beta = 0$, so $\alpha$ is closed. 
 
Next suppose that $\alpha$ is WB but not closed. Since $\alpha$ is not
closed, there exists a chain $A_k$ such that $\partial A_k=0$ but
$\int_{A_k} \alpha \ne 0$. But then $\left | \int_{A_k} \alpha \right
|$ is not bounded by $K |\partial A_k|$, so $\alpha$ is not
WB. Contradiction.
 \end{proof} 

 Since the coboundaries of strongly PE cochains are both WE and WB,
 cohomologous cochains are either both WE or neither WE, and are
 either both WB or neither. We can therefore speak of a cohomology
 {\em class} being AN or WB, and define subspaces $\check
 H^k_{AN}(\Omega_\T, \R)$ and $\check H^k_{WB}(\Omega_\T, \R)$ of
 $\check H^k(\Omega_\T, \R)$.

\medskip

The situation when $k=1$ is simple:

\begin{thm}\label{k_equals_1} If $\T$ is a repetitive aperiodic tiling of $\R^n$ with FLC, then every strongly PE and WB 1-cochain is WE. 
\end{thm}

\begin{proof} When $n=1$ this is a special case of the classical
  Gottschalk-Hedlund theorem \cite{GottschalkHedlund}. Many generalization for
  higher-dimensional actions have been proven over the years. A proof
  of this specific situation can be found in \cite{KS}.  \end{proof}

The situation when $k=0$ is even simpler. The only WB cochain is the
zero cochain, which is WE.

\medskip

A {\em stepped plane} is a canonical projection tiling from 3 to 2
dimensions. We use the term more generally for a canonical projection
tiling from $\R^{n+1}$ to $\R^n$.  Recall that an $m$ to $n$
dimensional projection tiling is {\em canonical} if the window in the
internal space $\R^{m-n}$ is the projection of a unit cube in $\R^m$.

\begin{thm}\label{stepped-thm} 
Let $\T$ be a stepped plane in $\R^n$, and let $\alpha$
  represent a strongly PE and WB class in $\check H^k(\Omega_\T,
  \R)$. Then $\alpha$ is WE.
\end{thm}

\begin{proof} Since $\T$ is a canonical cut-and-project tiling from
  $n+1$ to $n$ dimensions, the window for the projection is an
  interval whose endpoints correspond to the same hyperplane $H$ in
  the torus $T=\R^{n+1}/\Z^{n+1}$.  As a topological space,
  $\Omega_\T$ is then obtained from $T$ by removing $H$ and gluing in
  two copies of $H$, each representing a limit from one side.  The \v
  Cech cohomology of $\Omega_\T$ is then isomorphic to the cohomology
  of $T$ with one point removed. That is, $\check H^k(\Omega_\T,\R)$
  is isomorphic to $H^k(T,\R)$ for $k=0,\ldots, n$. In particular,
  $\check H^*(\Omega_\T, \R)$ is freely generated as an exterior
  algebra, in dimensions up through $n$, by $\check H^1(\Omega_\T,
  \R)$.

Kellendonk and Sadun \cite{KS2} proved that, for $m$-to-$n$
dimensional cut-and-project tilings with polyhedral windows, $\dim
\check H^1_{AN}(\Omega_\T, \R) = m-n$. For stepped planes, this means
that we can choose a basis for $\check H^1(\Omega_\T, \R)$ such that
one basis element can be represented, using the de Rham version of
cohomology, by a form $df$, where $f$ is a weakly PE function, and
that the other $n$ basis elements can be represented the constant
forms $dx^i$. A basis for $\check H^k$ is then given by the forms
$dx^I := dx^{i_1} \wedge \cdots \wedge dx^{i_k}$ for multiindices
$I=\{i_1,\ldots,i_k\}$, and $df \wedge dx^J$ for multiindices
$J=\{j_1, \ldots, j_{k-1}\}$.

Note that $df \wedge dx^J = d(f dx^J)$ is WE, and hence WB, 
while $dx^I$ (and all nonzero linear combinations of the $dx^I$'s) is
not WB, and hence not WE. This implies 
that $\check H^k_{AN}(\Omega_\T,\R)$ and $\check H^k_{WB}(\Omega_\T,\R)$
are both the span of the $df\wedge dx^J$'s, and hence are equal.  \end{proof}

Theorem \ref{Thm2} is just the $k=n$ case of Theorems 
\ref{k_equals_1} and \ref{stepped-thm}.

\section{Concluding Remarks}
There has been a burst of activity in recent years studying the
bounded displacement (BD) equivalence relation for tilings and Delone
sets
\cite{Haynes,HK,HKK,HKW,KellySadun,OpenProblems,Solomon,Solomon2}. As
we reported in our previous paper \cite{KellySadun}, there are many
relevant papers contributing to the subject that predate the
terminology BD such as \cite{DO2,DO1,Kesten}. The papers
\cite{DO2,DO1} in particular make the connection to quasicrystals
explicit. They studied the question of when a cut-and-project set is
BD to a crystal. When a mathematical quasicrystal can be written as a
small perturbation of a mathematical crystal (via a BD mapping), then
it may be possible that the quasicrystal can be constructed from the
crystal by a {\it displacive} phase transition. A closely related
notion to BD is that of a bounded remainder set (BRS). There has been
quite a bit of activity (including new cohomological work) in this
subject area in recent years
\cite{GrepLev,GrepLev2,GrepLev3,HK,HKK,KozmaLev,KellySadun}.\newline

In a recent work \cite{DKLL} proved many analogues of the Euclidean
results on the BD and BL equivalence relations for connected and
simply connected nilpotent Lie groups (with respect to both the
Riemannian and Carnot-Carth\'eodory metrics. Note that these groups
are topologically Euclidean.). The results in \cite{DKLL} seem to
indicate that Euclidean methods are fairly adaptable to the nilpotent
setting, as far as BD is concerned. However, we are not aware of any
work on the cohomology of tiling spaces for such groups. It would
be interesting if the results of the present article had analogues in
connected simply connected nilpotent Lie groups.


\end{document}